\documentclass[12pt,reqno]{amsart}
\usepackage{amsopn,amssymb,mathrsfs,skak,url,mathrsfs,a4wide,dsfont}

\newtheorem{thm}{Theorem}[section]
\newtheorem{cor}[thm]{Corollary}
\newtheorem{lem}[thm]{Lemma}
\newtheorem{prop}[thm]{Proposition}

\theoremstyle{definition}
\newtheorem{defn}[thm]{Definition}
\newtheorem{example}[thm]{Example}

\newcommand{\bdy}{\partial}                                         
\newcommand{\bish}{{\rm(\symbishop)}}                               
\newcommand{\cl}[1]{\ensuremath{\overline{{#1}}}}                   
\newcommand{\C}{\mathbb{C}}                                         
\newcommand{\map}[3]{\ensuremath{{#1}:{#2}\longrightarrow{#3}}}     
\newcommand{\N}{\mathbb{N}}                                         
\newcommand{\n}[1]{\ensuremath{\left\|{#1}\right\|}}                
\newcommand{\R}{\mathbb{R}}                                         
\newcommand{\res}[1]{\ensuremath{\!\!\upharpoonright_{#1}}}         
\newcommand{\set}[2]{\ensuremath{\left\{{#1}\;:\;\,{#2}\right\}}}   
\newcommand{\st}{($*$)}                                             
\newcommand{\wone}{\ensuremath{\omega_1}}                           

\DeclareMathOperator{\supp}{supp}                                   

\newcounter{smallromans}

\newenvironment{romanenumerate}
{\begin{list}{{\normalfont\textrm{(\roman{smallromans})}}}%
  {\usecounter{smallromans}\setlength{\itemindent}{0cm}%
   \setlength{\leftmargin}{5.5ex}\setlength{\labelwidth}{5.5ex}%
   \setlength{\topsep}{.5ex}\setlength{\partopsep}{.5ex}%
   \setlength{\itemsep}{0.1ex}}}%
{\end{list}}

\newcounter{smallromansdash}

{\end{list}}

\newcounter{bigromans} 
  {\end{list}}

\begin{document}
\title{Chains of functions in $C(K)$-spaces}
\begin{abstract}
The Bishop property {\bish}, introduced recently by K.\ P.\ Hart, T.\ Kochanek and the
first-named author, was motivated by Pe\l czy\'nski's classical work on weakly compact
operators on $C(K)$-spaces. This property asserts that certain chains of functions
in said spaces, with respect to a particular partial ordering, must be countable. There are two versions of
{\bish}:\ one applies to linear operators on $C(K)$-spaces and the other to the compact
Hausdorff spaces themselves.\smallskip

We answer two questions that arose after {\bish} was first introduced. 
We show that if $\mathscr{D}$ is a class of compact spaces that is preserved when taking closed subspaces and
Hausdorff quotients, and which contains no non-metrizable linearly ordered space, then
every member of $\mathscr{D}$ has {\bish}. Examples of such classes include all $K$
for which $C(K)$ is Lindel\"of in the topology of pointwise convergence (for instance,
all Corson compact spaces) and the class of Gruenhage compact spaces. We
also show that the set of operators on a $C(K)$-space satisfying {\bish} does not
form a right ideal in $\mathscr{B}(C(K))$. Some results regarding local connectedness
are also presented.
\end{abstract}

\author[T. Kania]{Tomasz Kania}
\address{Department of Mathematics and Statistics, Fylde College, Lancaster University,
Lancaster LA1 4YF, United Kingdom---and---Institute of Mathematics, Polish Academy of Sciences, \'{S}niadeckich 8, 00-056 Warszawa, Poland}
\email{tomasz.marcin.kania@gmail.com}

\author[R. J. Smith]{Richard J. Smith}
\address{School of Mathematical Sciences, University College Dublin, Belfield, Dublin 4, Ireland}
\email{richard.smith@maths.ucd.ie}

\thanks{This paper grew out of discussions held at the Workshop on set theoretic methods
in compact spaces and Banach spaces, Warsaw, and at the inaugural meeting of
the QOP network, Lancaster University, April 2013.}

\date{\today}
\maketitle

\section{Introduction}\label{intro}
The aim of this note is to continue the line of research undertaken in the recent work \cite{hkk:13} of Kochanek, Hart and the first-named author concerning the so-called {\em `Bishop' property} of Hausdorff spaces, denoted {\bish}, which arose from operator-theoretic considerations.\medskip

Let $K$ be a compact Hausdorff space and denote by $C(K)$ the Banach space of all scalar-valued continuous functions on $K$ furnished with the supremum norm. Pe\l czy\'nski characterized weakly compact operators $T\colon C(K)\to X$, where $X$ is an arbitrary Banach space, as precisely those which do not preserve copies of $c_0$ inside $C(K)$. In fact, $T$ is not weakly compact if it does not preserve copies of $c_0$ spanned by sequences of disjointly supported norm-one functions. Given a sequence $(f_n)_{n=1}^\infty$ of such functions in $C(K)$, set $g_n = f_1 + \cdots + f_n$ ($n\in \mathbb{N}$)---then the functions $(g_n)_{n=1}^\infty$ behave like elements of the summing basis of $c_0$. Therefore, we infer that the operator $T$ is weakly compact if and only if $\inf_{n>m}\|Tg_n - Tg_m\|=0$. Let us put this observation into a more general framework.\medskip

Given two distinct functions $f,g\in C(K)$, we write
\[
\text{$f\prec g$ whenever $f\res{\supp f} \;= g\res{\supp f}$}.
\]
Here, $\supp f$ denotes the closure of $\{t \in K\colon f(t)\neq 0\}$. The relation
$\prec$ is a partial ordering on $C(K)$, which has been already studied in the context of positive elements in arbitrary C*-algebras (see \cite[Theorem 3.11]{geometric}, where $\prec$ is called the \emph{geometric pre-ordering}). We shall however confine ourselves to the classical, commutative setting.\medskip

Let $X$ be a Banach space and let $\map{T}{C(K)}{X}$ be a bounded linear operator. Then $T$ is said to have {\bish} if
\[
\inf\{\n{Tf - Tg}\colon f,g \in F, f \prec g\} \;=\; 0,
\]
whenever $F$ is a norm-bounded uncountable $\prec$-chain in $C(K)$ ($F$ is a $\prec$-chain if for all distinct $f,g\in F$ either $f\prec g$ or $g\prec f$). Using this ostensibly \emph{ad hoc} definition we can rephrase the above-mentioned theorem of Pe\l czy\'nski: the operator $T\colon C(K) \to X$ is weakly compact if and only if $\inf\{\n{Tf - Tg}\colon f,g \in F, f \prec g\}=0$ for every norm-bounded countable $\prec$-chain $F$ in $C(K)$. It was proved in \cite{hkk:13} that if $K$ is extremally disconnected compact Hausdorff space, then $T$ is weakly compact if and only if it has {\bish}. Because the identity operator on a $C(K)$-space, $I_{C(K)}$, is never weakly compact (unless $K$ is finite), we can ask what topological properties of $K$ allow $I_{C(K)}$ to have {\bish} (in this case we say that $K$ itself has {\bish}). The class of compact spaces having {\bish} with this respect can be thought of as far distant as possible from the class of extremally disconnected compact spaces and, on the other hand, it is a common roof for the classes of compact metric spaces (as we shall now explain) and locally connected compact spaces.\medskip

Before we explain why compact metric spaces have {\bish}, let us reformulate this property in the following helpful way. Given $\delta>0$ and $f,g\in C(K)$, we write $f\prec_\delta g$ if $f\prec g$ and
$\n{f-g}\geqslant \delta$. A subset $F\subseteq C(K)$ is called a {\em $\delta$-$\prec$-chain}
if, for any two different $f,g\in F$, either $f\prec_\delta g$ or $g\prec_\delta f$.
Thus, $K$ has {\bish} if and only if, for each $\delta>0$, every bounded $\delta$-$\prec$-chain
in $C(K)$ is at most countable. By rescaling, it follows that $K$ has {\bish} if and only if
every bounded $1$-$\prec$-chain in $C(K)$ is at most countable.\medskip

Apparently every compact metric space $K$ enjoys this property, as in this case $C(K)$ is
separable in the norm topology, and hence contains no uncountable discrete subset (evidently,
any $1$-$\prec$-chain is discrete in the norm topology). On the other hand, the ordinal interval $[0,\wone]$, the lexicographically ordered split interval $[0,1]\times\{0,1\}$ and
the \v Cech--Stone compactification of the natural numbers $\beta\N$ are examples of compact
spaces that do not have {\bish}.  To see that the first two
spaces do not have {\bish}, consider the uncountable 1-$\prec$-chains of indicator functions
$\set{\mathds{1}_{[0,\alpha]}}{\alpha < \wone}$ and $\set{\mathds{1}_{[(0,0),(x,0)]}}{x \in [0,1]}$
in the corresponding spaces of continuous functions, respectively. In the case of $\beta\N$,
consider an enumeration of the rational numbers $(q_n)_{n=1}^\infty$, the sets $E_x =\set{n \in \N}{q_n < x}$,
$x \in \R$, and finally the functions $\mathds{1}_{E_x} \in \ell_\infty$ and their canonical
extensions to $\beta\N$. (Proposition~\ref{tych} generalizes this to the \v Cech--Stone compactifications of Tychonoff spaces from a wider class.) \medskip

Given these examples, a compact space $K$ may be viewed as being in some way
well-behaved if it has {\bish}. Thus, we find the task of identifying classes of
compact spaces having {\bish} to be natural and important, both in terms of
the topology of compact spaces $K$ and the ideal structure of the Banach algebra
$\mathscr{B}(C(K))$.\medskip

Let $\mathscr{L}$ denote the class of compact spaces $K$ for which $C_p(K)$ is Lindel\"of,
where $C_p(K)$ denotes $C(K)$ in the topology of pointwise convergence. As far as the authors
are aware, $\mathscr{L}$ has not been fully delineated. However, it is known that $\mathscr{L}$ contains the
important subclass $\mathscr{C}$ of Corson compact spaces \cite{alster-pol:80,gulko:78}.

\begin{defn}\label{corson}
A compact space $K$ is called {\em Corson} if, for some set $\Gamma$, it is homeomorphic to a
subspace of $\Sigma(\Gamma)$ in the pointwise topology, where
\[
\Sigma(\Gamma) \;=\; \set{f \in \R^\Gamma}{\text{$f(\gamma)\neq 0$ for at most countably many $\gamma \in \Gamma$}}.
\]
\end{defn}

All metrizable, Eberlein, Talagrand and Gul'ko compact spaces are in $\mathscr{C}$.
In this note we show that all spaces in $\mathscr{L}$ have {\bish}. Thus
we answer positively \cite[Question 3.9]{hkk:13}, which asks whether Eberlein compact spaces
(spaces homeomorphic to weakly compact subsets of Banach spaces) have {\bish}.\medskip

There is another large, though lesser-known, class of compact spaces of relevance to this note.
It was first introduced in \cite{gru:87}, and the second-named author found it to be of importance
when studying strictly convex norms on Banach spaces. The definition
below is equivalent to that given in \cite{gru:87} -- see \cite[Proposition 2]{smith:09}.

\begin{defn}\label{gru}
We say that a compact space $K$ is {\em Gruenhage} 
if we can find a sequence $(\mathscr{U}_n)_{n=1}^\infty$ of families of open subsets of $K$,
together with a sequence of open subsets $(R_n)_{n=1}^\infty$ of $K$, such that
\begin{enumerate}
\item $U \cap V = R_n$ whenever $n\in \N$ and $U,V \in \mathscr{U}_n$ are distinct, and
\item if $s,t \in K$, then $\{s,t\} \cap U$ is a singleton for some
$m \in \N$ and some $U \in \mathscr{U}_m$.
\end{enumerate}
\end{defn}

Let us denote by $\mathscr{G}$ the class of Gruenhage compact spaces. All metrizable, Eberlein,
Gul'ko and descriptive compact spaces are in $\mathscr{G}$. In particular,
all scattered compact spaces having countable Cantor--Bendixson height or, more generally, all compact
$\sigma$-discrete spaces (unions of countably many relatively discrete subsets), are descriptive
and thus members of $\mathscr{G}$. We prove that all spaces in $\mathscr{G}$ have {\bish}.\medskip

That all elements of $\mathscr{L}$ and $\mathscr{G}$ have {\bish} follows from the
next result.

\begin{thm}\label{main1}
Suppose that $\mathscr{D}$ is a class of compact Hausdorff spaces that is preserved when
taking closed subspaces and Hausdorff quotients, and which contains no non-metrizable linearly
ordered space. Then every member of $\mathscr{D}$ has {\bish}.
\end{thm}

It follows from the Tietze--Urysohn extension theorem and the Hahn--Banach theorem, respectively, that $\mathscr{L}$ is preserved when taking closed subspaces
and Hausdorff quotients. It was proved in \cite{nakhmanson:88} that $\mathscr{L}$ contains no
non-metrizable linearly ordered elements. Regarding $\mathscr{G}$, it is immediate that
this class is preserved under closed subspaces. Preservation under continuous images is proved
in \cite[Theorem 23]{smith:09}, and the fact that $\mathscr{G}$ contains no non-metrizable linearly ordered
elements follows from \cite[Proposition 6.5]{bl:13}.\medskip

It is worth noting that $\mathscr{L}$ and $\mathscr{G}$ are incomparable.
Mr\'owka space $\Psi$, defined using a~maximal, almost disjoint family of subsets of $\N$, is a
compact scattered space of Cantor--Bendixson height $3$, so is Gruenhage. However, $C_p(\Psi)$ is not Lindel\"of \cite[Proposition 1]{dowsimon:06}. On the other hand, there is a
Corson compact space that does not contain any dense metrizable subset \cite[p.\ 258]{tod:81}, and
every Gruenhage compact space possesses such a subset \cite[Theorem 1]{gru:87}.\medskip

Section \ref{main1proof} is devoted to proving Theorem \ref{main1}. Section \ref{connect}
explores {\bish} in the context of connected and locally connected spaces. In Section \ref{further-observations},
we answer in the negative Question 4.3 of \cite{hkk:13}.

\section{The proof of Theorem \ref{main1}}\label{main1proof}

Before proceeding with the proof, we introduce some notation and auxiliary results.
Given a linearly ordered set $F$ and $f, g \in F$, we define the intervals $(f,g)$, $(f,g]$,
$[f,g]$ and $[f,g)$ in the obvious way. We let $(\leftarrow,f)$ and $(f,\rightarrow)$ denote
the set of strict predecessors and strict successors of $f$, respectively, and define
$(\leftarrow,f]$ and $[f,\rightarrow)$ accordingly.\medskip

A subset $I \subseteq F$ is called an {\em initial segment} if $f \prec g$ and $g \in I$ implies
$f \in I$. The set $\mathscr{I}$ of initial segments of $F$ is naturally linearly ordered with
respect to inclusion, and is compact with respect to the induced order topology.\medskip

Let $K$ be a compact Hausdorff space and let us fix a non-empty $1$-$\prec$-chain $F \subseteq C(K)$. Set
$D=\set{z \in \C}{|z|\geqslant 1}$. Given $I \in \mathscr{I}$, we define
\[
W_I \;:=\; \bigcap_{f \in I} f^{-1}(0) \cap \bigcap_{g \in F\setminus I} g^{-1}(D),
\]
where $f^{-1}(0)$ is shorthand for $f^{-1}(\{0\})$. The next proposition lists some
straightforward yet important facts about the $W_I$ to be used in the proof of Theorem \ref{main1}.

\begin{prop}\label{initseg} If we fix a $1$-$\prec$-chain $F\subseteq C(K)$, then the following statements hold. Throughout, $I$ and $J$ are assumed to be elements of $\mathscr{I}$.
\begin{romanenumerate}
\item If $\varnothing \neq I \subsetneq F$, then $W_I$ is non-empty.
\item If $I \subsetneq J\subseteq F$, then $W_I \cap W_J$ is empty.
\item If $P \in \mathscr{I}$, then $\bigcup_{I \subseteq P} W_I$ and $\bigcup_{P \subseteq J} W_J$ are compact.
\end{romanenumerate}
\end{prop}\medskip

\begin{proof}~
\begin{romanenumerate}
\item If $f \prec_1 g$, then $f^{-1}(0) \cap g^{-1}(D)$ is non-empty. Bearing
this in mind, if $\varnothing \neq I \subsetneq F$, then $W_I$ is non-empty, being as it is the
intersection of a family of non-empty compact sets having the finite intersection property.
\item Given $g \in J\setminus I$,
we have $W_I \subseteq g^{-1}(D)$ and $W_J \subseteq g^{-1}(0)$.
\item Let $P\in\mathscr I$. We show that $\bigcup_{I \subseteq P} W_I$ is compact. First, we assume that $P \neq F$. We claim that
\[
\bigcup_{I \subseteq P} W_I \;=\; \bigcap_{g \in F\setminus P} g^{-1}(D),
\]
which is of course compact. By definition, if $I \subseteq P$ and $g \in F\setminus P$, then $W_I$ is a~subset of $g^{-1}(D)$.
To see the other inclusion, fix $$t \in \bigcap_{g \in F\setminus P} g^{-1}(D)$$ and
set
\[
I \;=\; \set{g \in P}{|g(t)| < 1}.
\]
Using the definition of $\prec$, it follows that $I$ is an initial segment. We claim that $t \in W_I$.
Trivially, $|g(t)| \geqslant 1$ whenever $g \in F\setminus I$. Moreover, if $f \in I$ then $f(t)=0$.
Indeed, assume that $f(t)\neq0$. Pick $g\in F\setminus I$ (which we can do as $I \subseteq P \neq F$). Then $f\prec g$, and so $|f(t)|
=|g(t)|\geqslant 1$; hence $f\not\in I$.
Thus $t \in W_I$ as claimed.

\smallskip In the case where $P=F$, we proceed differently. Suppose that $\mathscr{U}$ is an open cover of $\bigcup_{I \subseteq P} W_I$. Then because $W_F$ is compact, there is a finite family $\mathscr{G}\subseteq \mathscr{U}$ satisfying $W_F \subseteq \bigcup\mathscr{G}$.
According to the definition of $W_F$, there exists $f \in F$ such that $f^{-1}(0) \subseteq \bigcup\mathscr{G}$.
Setting $Q=(\leftarrow,f)$, we get $Q \neq F$ and so $\bigcup_{I \subseteq Q} W_I$ is compact, from above.
Thus $$\bigcup_{I \subseteq Q} W_I \subseteq \bigcup\mathscr{F}$$ for some finite family $\mathscr{F} \subseteq \mathscr{U}$.
We conclude that
\[
\bigcup_{I \subseteq F} W_I \;\subseteq\; \bigcup(\mathscr{F} \cup \mathscr{G}),
\]
since $I \not\subseteq Q$ implies that $f \in I$, meaning $W_I \subseteq f^{-1}(0) \subseteq \bigcup\mathscr{G}$.

\smallskip Now we show that $\bigcup_{P \subseteq J} W_J$ is compact. Consider a new $1$-$\prec$-chain $G = F\setminus P$,
its set of initial segments $\mathscr{J}$, together with the sets
\[
W'_I \;=\; \bigcap_{f \in I} f^{-1}(0) \cap \bigcap_{g \in G\setminus I} g^{-1}(D),
\]
where $I \in \mathscr{J}$. Observe that the map $J \mapsto J\setminus P$ is a bijection between the set of $J \in \mathscr{I}$
containing $P$ and $\mathscr{J}$, and $W_J = \bigcap_{f \in P} f^{-1}(0) \cap W'_{J\setminus P}$. Therefore,
\[
\bigcup_{P \subseteq J} W_J \;=\; \bigcap_{f \in P} f^{-1}(0) \cap \bigg( \bigcup_{I \in \mathscr{J}} W'_I \bigg),
\]
so it is compact, from above. \qedhere
\end{romanenumerate}
\end{proof}

We remark in passing that if $W_\varnothing$ is empty then $F$ has a least element, and if $W_F$ is empty then $F$ has a greatest element.
Indeed, if $W_F$ is empty then $f^{-1}(0)$ must be empty for some $f \in F$. However, the definition of $\prec$
forces any such $f$ to be the greatest element of $F$. Likewise, if $W_\varnothing$ is empty, then $g(K) \subseteq D$
for some $g \in F$, and any such $g$ is necessarily the least element of a $1$-$\prec$-chain.\medskip

We also need the following result, which belongs to the folklore of linearly ordered topological spaces.

\begin{prop}\label{countablegaps}
Let $X$ be a second-countable linearly ordered topological space, and let $A$
be the set of $x \in X$ such that $(x,\to)$ has a least element. Then $A$ must be countable.
\end{prop}

\begin{proof}
Assume that $A$ is uncountable. Let $\mathscr{V}$ be a countable base for $X$.
Since each $(\leftarrow,x]$, $x \in A$, is open, there is an uncountable set $B \subseteq A$ and $V \in \mathscr{V}$, such that
$x \in V \subseteq (\leftarrow,x]$ for all $x \in B$, but applying this to any distinct $x,y \in B$ yields a contradiction.
\end{proof}

Now we are able to prove Theorem \ref{main1}.

\begin{proof}[Proof of Theorem \ref{main1}]
Let $K \in \mathscr{D}$, where $\mathscr{D}$ is a class satisfying the hypotheses of the theorem.
Let $F$ be a $1$-$\prec$-chain in $C(K)$, and set $W = \bigcup_{I \in \mathscr{I}} W_I$.
From Proposition \ref{initseg} (iii), we know that $W$ is compact, so in particular $W \in \mathscr{D}$.
Define the map
$\map{\pi}{W}{\mathscr{I}}$ by $\pi(t)=I$ whenever $t \in W_I$ and $I\in \mathscr{I}$. Notice that $\mathscr{I}\setminus \pi(W) \subseteq \{\varnothing,\,F\}$,
by Proposition \ref{initseg} (i). We know that $\pi$ is
continuous because given an open interval $(P,Q) \subseteq \mathscr{I}$,
\[
\pi^{-1}\big((P,Q)\big) \;=\; W\setminus\bigg(\bigg(\bigcup_{I \subseteq P} W_I\bigg)\cup \bigg(\bigcup_{Q \subseteq J} W_J\bigg)\bigg),
\]
is open in $W$, again by Proposition \ref{initseg} (iii). Therefore, $\pi(W) \in \mathscr{D}$ as well.
Now $\pi(W)$ is a~closed interval in the compact linearly ordered space $\mathscr{I}$, so $\pi(W)$
must be metrizable, by hypothesis. It follows that $\mathscr{I}$ is also metrizable.
Finally, given $f \in F$, the initial segment $(\leftarrow,f)\in \mathscr{I}$ has immediate successor
$(\leftarrow,f]\in\mathscr{I}$, so $F$ must be countable, by Proposition \ref{countablegaps}.
\end{proof}

We identify a further class of compact spaces having {\bish}. As with Gruenhage
spaces, the next property was studied in the context of strictly convex norms on Banach spaces \cite{ost:12}.

\begin{defn}\label{star}
We say that a compact space $K$ has {\st} if we can find a sequence $(\mathscr{U}_n)_{n=1}^\infty$ of families of open subsets of $K$,
with the property that given any $x,y \in K$, there exists $n$ such that
\begin{enumerate}
\item $\{x,y\} \cap \bigcup\mathscr{U}_n$ is non-empty, and
\item $\{x,y\} \cap U$ is at most a singleton for all $U \in \mathscr{U}_n$.
\end{enumerate}
\end{defn}

Every Gruenhage compact space has {\st} \cite[Proposition 4.1]{ost:12}, but there are examples of compact scattered
non-Gruenhage spaces having {\st}, both in ZFC and elsewhere (see \cite{smith:12} and \cite[Example 2]{ost:12}, respectively). 
Every scattered compact space having
{\st} has {\bish}, because again the class of such things satisfies the hypotheses of
Theorem \ref{main1}, by  \cite[Proposition 4.5]{ost:12} and \cite[Proposition 6.5]{bl:13}.

\section{Connectedness, local connectedness and their effects on {\bish}}\label{connect}

Drawing pictures of sequences of bumps will suggest to the reader that some form of connectedness
will have consequences for {\bish}.
The next proposition and example, which generalize the fact that $C(\beta\N)$ does not have {\bish},
shows that standard connectedness does not force spaces to have {\bish}.

\begin{prop}\label{tych}
Let $X$ be a Tychonoff (i.e.\ completely regular) space that admits a~countable and locally finite family
$\mathscr{U}$ of pairwise disjoint non-empty open sets. Then neither $\beta X$
nor $\beta X\setminus X$ has {\bish}.
\end{prop}

\begin{proof} Fix an enumeration $U_n$, $n \in \N$, of $\mathscr{U}$. Let $x_n \in U_n$ and take continuous functions $\map{f_n}{X}{[0,1]}$ such that
$f_n(x_n)=1$ and $f_n$ vanishes on $X\setminus U_n$ ($n\in \N$). Let $(q_n)_{n=1}^\infty$ be
an enumeration of the rationals and, given $x \in \R$, define $E_x = \set{n \in\N}{q_n<x}$.
As $\mathscr{U}$ is locally finite,
\[
g_x = \sum_{n \in E_x} f_n\quad (x\in \R),
\]
is a well-defined, continuous and bounded function. Let $\cl{g_x}$, denote the continuous extension of $g_x$ to $\beta X$, where $x \in \R$.
Suppose that $x<y$. Then the set $E_y\setminus E_x$ is infinite. If $p \in \beta X$ is any limit point of
$\set{x_n}{n \in E_y\setminus E_x}$, then necessarily $p \notin X$, because $\mathscr{U}$
is locally finite. It is evident that $\cl{g_x}(p)=0$ and $\cl{g_y}(p)=1$. Therefore
the $\cl{g_x}$, $x\in\R$, and also their restrictions to $\beta X\setminus X$, form uncountable
$1$-$\prec$-chains in $C(\beta X)$ and $C(\beta X\setminus X)$, respectively.
\end{proof}

\begin{cor}\label{connectednobishop}
The spaces $\beta\R$ and $\beta\R\setminus\R$ do not have {\bish}, according to Proposition \ref{tych}.
\end{cor}

On the other hand, we can formulate a sufficient condition if we consider a certain
type of {\em local} connectedness.
Before proving our main result of this section, Theorem \ref{locconbishop}, we make some
preparatory observations. The next definition captures the precise notion of local connectedness that
we require.

\begin{defn}\label{localrelativecon}
Given a closed subset $M$ of a compact space $K$, and $t \in M$, we say that
{\em $t$ has a local base of connected sets relative to $M$} if, given any set $U \ni t$
open in $M$, there exists a connected set $V$, also open in $M$, and satisfying $t \in V \subseteq U$.
\end{defn}

Clearly, $K$ is locally connected if every point of $K$ has a local base of connected
sets relative to $K$. Given $E \subseteq K$, let $\bdy E$ denote the (possibly empty)
boundary of $E$. We recall that, if a subset $V \subseteq K$ is connected
and $U\subseteq K$ is an open set such that $V \cap U$ and $V\setminus U$ are non-empty, then $V \cap \bdy U$ is non-empty.

\begin{lem}\label{1-step}
Suppose that $M \subseteq K$ is closed and we have elements $f_n$, $n\in \mathbb{N}$, and $f$
in $C(K)$, satisfying $f_1 \prec f_2 \prec f_3 \prec \dots \prec f$, and $t_n \in M$, $n\in \N$, such that $f_n(t_n)=0$ and $|f_{n+1}(t_n)|\geqslant 1$ for all $n$. Then, if $u$ is any accumulation point of the sequence $(t_n)_{n=1}^\infty$, then $u$ does not have a local base of
connected sets relative to $M$.
\end{lem}

\begin{proof}
Let $U_n = \set{s \in K}{|f_{n+1}(s)| > \frac{1}{2}}$, $n\in \N$. Evidently, $|f_{n+1}(s)|=\frac{1}{2}$ if $s \in \bdy U_n$. Suppose that $u \in M$ is an accumulation point of the sequence $(t_n)_{n=1}^\infty$. Then
$u \notin U_n$ for any $n$, because $m>n$ implies that $f_n \prec f_m$, giving $f_{n+1}(t_m)=0$ and $t_m \notin U_n$.
We claim that $|f(u)| \geqslant 1$. Indeed, otherwise, we can find an open set $U \ni u$
such that $|f(s)| < 1$ whenever $s \in U$. But this implies that $t_n \in U$ for some $n$, and since
$f_{n+1}\prec f$, we have $1 \leqslant |f_{n+1}(t_n)| = |f(t_n)| < 1$. Now
let $U=\set{s \in K}{|f(s)|>\frac{1}{2}}$. If $u$ does have a local base of connected
sets relative to $M$, then we could find a connected set $V$, open in $M$, such that $u \in V \subseteq U \cap M$.
However, given $m$ satisfying $t_m \in V$, we have $t_m \in V \cap U_m$ and $u \in V\setminus U_m$.
By connectedness, there exists some $s \in V \cap \bdy U_m$, giving $\frac{1}{2} = |f_{m+1}(s)|= |f(s)| > \frac{1}{2}$,
which is a contradiction.
\end{proof}

Now suppose that we have a bounded $\prec$-chain $F \subseteq C(K)$.
Given a~closed set $M\subseteq K$, we define an equivalence relation $\sim_M$ on $F$ by declaring that
$f \sim_M g$ if and only if $\n{(f-g)\res{M}} < 1$. Evidently, $\sim_M$ is reflexive
and symmetric. To obtain transitivity, notice
that if $t \in K$ and $f \prec g \prec h$, then either $g(t)=f(t)$ or $g(t)=h(t)$
(if $g(t)\neq f(t)$ then $g(t)\neq 0$, giving $g(t)=h(t)$). Moreover, the equivalence
classes of $\sim_M$ are intervals in $(F, \prec)$.\medskip

The main result of this section now follows.
Recall that a linear ordering is {\em scattered} if it contains no order-isomorphic copies
of the rationals.

\begin{thm}\label{locconbishop}
Let $M$ be closed a closed subset of a compact, Hausdorff space $K$ and suppose that we can write the remainder $K\setminus M$ as a union $\bigcup_{n=1}^\infty H_n$,
where each $H_n$ is open in $\cl{H_n}$ {\rm (}$n\in \N${\rm )}, and every
point of $H_n$ has a base of neighbourhoods that are connected sets relative to $\cl{H_n}$. Then the
following statements hold.
\begin{enumerate}
\item Every equivalence class, with respect to $\sim_M$, of any given $1$-$\prec$-chain, is countable and scattered with respect to the induced ordering.
\item If $M$ has {\bish}, then so does $K$.
\item In particular, if $M$ is empty then $K$ has {\bish}.
\end{enumerate}
\end{thm}

Of course, if $K$ is locally connected, then Theorem \ref{locconbishop} shows that
$K$ has {\bish}. However, if $A$ is a locally connected and compact space, and $K$ is $\sigma$-discrete, i.e.,\ $K=\bigcup_{n=1}^\infty D_n$, where each $D_n$ is discrete in the relative topology, then the locally connected set $A \times \{t\}$ is open in $\cl{A \times D_n}^{A \times K}$ for all $t \in D_n$. Thus Theorem \ref{locconbishop} tells us that $A \times K$
has {\bish}. This example illustrates the fact that Theorem \ref{locconbishop}
can be applied to spaces that are rather far from being locally connected. Let us record the following corollary of Theorem~\ref{locconbishop}.

\begin{cor}\label{dualball}
Let $X$ be a Banach space and denote by $B_{X^*}$ the unit ball of $X^*$ endowed with the weak*-topology. Then $B_{X^*}$ has {\bish}.
\end{cor}

To prove Theorem \ref{locconbishop}, we require some machinery that is based on Haydon's analysis of
locally uniformly rotund norms on $C(K)$, where $K$ is a so-called \emph{Namioka--Phelps compact space} \cite{haydon:08}. Lemma \ref{minimal} below is essentially due to him. Let $X$ be a Hausdorff space, such
that $X=\bigcup_{n=1}^\infty H_n$, where each $H_n$ is open in its closure. Let
\[
\Sigma \;=\; \set{\sigma=(n_1,n_2,\dots,n_k)}{n_1 < n_2 < \dots < n_k,\, k \in \N}.
\]
We introduce a total ordering $\sqsubset$ on $\Sigma$ by declaring that $\sigma \sqsubset \sigma'$ if and only if $\sigma$ is a~proper extension of
$\sigma'$, or if there exists $k \in \N$ such that the $i^{{\rm th}}$ entries
$n_i$ and $n'_i$ of $\sigma$ and $\sigma'$, respectively, are defined for $i \leqslant k$, agree whenever $i < k$, and $n_k < n'_k$.
This is the {\em Kleene--Brouwer ordering} on $\Sigma$, and is different from the lexicographic ordering.\medskip

Given $\sigma=(n_1,n_2,\dots,n_k) \in \Sigma$, let
\begin{align*}
H_\sigma &\;=\; (\cl{H_{n_1}}\setminus H_{n_1}) \cap \dots \cap (\cl{H_{n_{k-1}}}\setminus H_{n_{k-1}}) \cap H_{n_k}\\
\text{and}\quad \widehat{H}_\sigma &\;=\; (\cl{H_{n_1}}\setminus H_{n_1}) \cap \dots \cap (\cl{H_{n_{k-1}}}\setminus H_{n_{k-1}}) \cap \cl{H_{n_k}}.
\end{align*}
Evidently, $H_\sigma \subseteq \widehat{H}_\sigma$ and $\widehat{H}_\sigma$ is closed.

\begin{lem}[{\emph{cf.} \cite[Lemma 3.3]{haydon:08}}]\label{minimal}
Let $W \subseteq X$ be a non-empty and compact set. Then there exists a minimal element $\sigma \in \Sigma$
such that $W \cap \widehat{H}_\sigma$ is non-empty. Moreover, for this $\sigma$, we have $W \cap H_\sigma = W \cap \widehat{H}_\sigma$.
\end{lem}

\begin{proof}
Let $n_1 \in \N$ be minimal, such that $W \cap \cl{H_{n_1}} \neq \varnothing$. If
$W \cap H_{n_1} = W \cap \cl{H_{n_1}}$, stop by setting $\sigma=(n_1)$. Else,
$W \cap (\cl{H_{n_1}}\setminus H_{n_1}) \neq \varnothing$, so let $n_2$ be minimal, such that $n_2 > n_1$ and
$W \cap (\cl{H_{n_1}}\setminus H_{n_1}) \cap \cl{H_{n_2}} \neq \varnothing$.
If $$W \cap (\cl{H_{n_1}}\setminus H_{n_1}) \cap H_{n_2} = W \cap (\cl{H_{n_1}}\setminus H_{n_1}) \cap \cl{H_{n_2}},$$
let us stop by setting $\sigma=(n_1,n_2)$. Otherwise, let $n_3$ be minimal, such that $n_3>n_2$ and
$$W \cap (\cl{H_{n_1}}\setminus H_{n_1}) \cap (\cl{H_{n_2}}\setminus H_{n_2}) \cap \cl{H_{n_3}} \neq \varnothing.$$
Continuing in this way, we have to stop after finitely many steps. Otherwise, we would get a
strictly increasing sequence $n_1 < n_2 < n_3 < \dots$ such that
\[
W_k \;:=\; W \cap \bigcap_{i=1}^k (\cl{H_{n_i}}\setminus H_{n_i}) \;\neq\; \varnothing,
\]
for all $k\in\N$. Set $V := \bigcap_{k=1}^\infty W_k$. Then $V\neq\varnothing$ by compactness. Let $j$
be minimal, subject to $V \cap H_j \neq \varnothing$, and let $k$ be such that
$n_k \leqslant j < n_{k+1}$ (it is clear that $n_1 \leqslant j$). Then we have
\[
\varnothing \;\neq\; V \cap H_j \;\subseteq\;W \cap \bigcap_{i=1}^k (\cl{H_{n_i}}\setminus H_{n_i}) \cap H_j.
\]
Evidently, this means that $n_k < j$, but this fact contradicts the minimal choice of
$n_{k+1}$.\medskip

So suppose that we have $\sigma$ determined as above. Now let $\sigma' \sqsubset \sigma$. If
$\sigma'$ properly extends $\sigma$, then $W \cap H_{\sigma'} = \varnothing$ by construction.
If $\sigma'$ is not a proper extension, take $k$ such that $n'_i= n_i$ for $i<k$ and $n'_k < n_k$
(where $n'_i$ denotes the $i^{{\rm th}}$ entry of $\sigma'$). Since
$n_k$ was chosen minimally so that $n_k > n_{k-1}$ and
\[
W \cap \bigcap_{i<k} (\cl{H_{n_i}}\setminus H_{n_i}) \cap \cl{H_{n_k}} \;\neq\; \varnothing,
\]
we must have
\[
W \cap \bigcap_{i<k} (\cl{H_{n_i}}\setminus H_{n_i}) \cap \cl{H_{n'_k}} \;=\; \varnothing.
\]
As $\widehat{H}_{\sigma'} \subseteq \bigcap_{i<k} (\cl{H_{n_i}}\setminus H_{n_i}) \cap \cl{H_{n'_k}}$, we conclude that $W \cap \widehat{H}_{\sigma'} = \varnothing$. The second assertion of the lemma is evident.
\end{proof}

\begin{proof}[The proof of Theorem \ref{locconbishop}]
Let us prove assertion (1). Suppose that we have a $1$-$\prec$-chain of $C(K)$, and let $F$
be an equivalence class of this chain with respect to $\sim_M$. We want to show
that $F$ is countable and scattered. This is done in two steps.\medskip

In the first step, we argue by contradiction to eliminate the possibility that $F$
contains an order isomorphic copy of $\wone$ or $\wone^*$. (Here, $\wone^*$ stands for $\omega_1$ with the reversed order.) If $F$ contains a copy $G$
of $\wone^*$, let $g$ be the greatest element of $G$. Then it is a straightforward
exercise to check that the set $\set{g-f}{f \in G}$ is a
$1$-$\prec$-chain which is, moreover, order-isomorphic to $\wone$.
Furthermore, it is easy to see that the elements of this new chain are $\sim_M$-equivalent.
Thus, if $F$ contains an isomorphic copy of $\wone$ or $\wone^*$, we can extract
$\sim_M$-equivalent elements $f_\alpha \in F$, $\alpha<\wone$, such that
$f_\alpha \prec_1 f_\beta$ whenever $\alpha<\beta$.\medskip

Recall that the sets $W_I$ introduced in Section \ref{main1proof}. Here, we define the non-empty compact set
\[
W_\alpha \;:=\; \bigcap_{\xi<\alpha} f_\xi^{-1}(\{0\}) \cap f_\alpha^{-1}(D),
\]
where $D$ is as in Section \ref{main1proof}. Observe that as the $f_\alpha$, $\alpha<\omega_1$, are $\sim_M$-equivalent, we have $W_\alpha \subseteq K\setminus M$. By applying Lemma
\ref{minimal} to $X:=K\setminus M$ and the $W:=W_\alpha$, for each $\alpha$ we obtain $\sigma_\alpha \in \Sigma$
satisfying
\[
W_\alpha \cap H_{\sigma_\alpha} \;=\; W_\alpha \cap \widehat{H}_{\sigma_\alpha} \;\neq\; \varnothing.
\]
Let $S_\sigma = \set{\alpha<\wone}{\sigma_\alpha = \sigma}$. Then $\wone = \bigcup_{\sigma\in\Sigma} S_\sigma$,
which implies that $S:=S_\sigma$ is stationary for some $\sigma\in \Sigma$ (i.e.\ $S$ meets every closed and unbounded subset of $\wone$; the implication follows from \cite[Theorem.~8.3]{jech}), which we fix for the remainder
of step one. As $S$
is stationary, we can find a strictly increasing sequence $(\beta_n)_{n=1}^\infty$ in $S$ which converges
to some $\beta \in S$. Indeed, if $L$ denotes the set of accumulation points (in $\wone$) of elements of $S$, then $L$ is closed and unbounded, thus there exists $\beta \in S \cap L$, from which the existence of $(\beta_n)_{n=1}^\infty$ follows.\medskip

Write $\sigma = (n_1,n_2,\dots,n_k)$, $m=n_k$, and set $A = \bigcap_{i<k} (\cl{H_{n_i}}\setminus H_{n_i})$,
so that
\[
W_\alpha \cap A \cap H_m \;=\; W_\alpha \cap A \cap \cl{H_m} \;\neq\; \varnothing,
\]
for every $\alpha \in S$. For each $n$, select $t_n \in W_{\beta_{n+1}} \cap A \cap H_m$.
We have $|f_{\beta_{n+1}}(t_n)| \geqslant 1$ and $f_{\beta_n}(t_n)=0$.
Let $u \in A \cap \cl{H_m}$ be an accumulation point of the $t_n$. Because $f_{\beta_{n+1}}\prec f_\beta$,
we have $|f_\beta(t_n)|\geqslant 1$ for all $n$, and so $|f_\beta(u)|\geqslant 1$. On the other hand,
if $\xi < \beta$ then there exists $N$ for which $\xi < \beta_n$ whenever $n \geqslant N$,
meaning that $f_\xi(t_n)=0$ for such $n$, and thus $f_\xi(u)=0$.
Therefore $u \in W_\beta$, and thus $u \in W_\beta \cap A \cap \cl{H_m} = W_\beta \cap A \cap H_m$.
However, according to Lemma \ref{1-step}, $u \in \cl{H_m}$ does not have a local base
of connected sets relative to $\cl{H_m}$, meaning that $u \notin H_m$. This contradiction completes
the first step.\medskip

We proceed with step two. We know that $F$ cannot contain an isomorphic copy of $\wone$ or $\wone^*$.
According to results that go back to Hausdorff, if we define a new equivalence relation $\sim$
on $F$ by $f\sim g$ whenever the interval $(f,g)$ is scattered, then every equivalence class
of $\sim$ is a scattered interval. Moreover, the quotient $F/\!\!\sim$, when endowed with
the induced order, is densely ordered. Again, according to Hausdorff, any uncountable scattered
order contains a copy of $\wone$ or $\wone^*$ (see \cite[Theorem 5.28]{rosenstein}). Since we have excluded this possibility, we
conclude that all equivalence classes of $\sim$ are countable.\medskip

The purpose of step two is to show that the quotient $F/\!\!\sim$ is in fact a singleton. From this
we conclude that $F$ is countable and scattered. We assume that $F/\!\!\sim$ is not a singleton
and reach a contradiction. Let $G \subseteq F$ have
the property that $G$ contains precisely one element of each equivalence class of $\sim$.
Then $G$ is densely ordered when given the induced order:\ if $f,h \in G$ and $f \prec h$,
then $f \prec g \prec h$ for some $g \in G$.\medskip

Consider the set $\mathscr{D}$ of all initial segments of $G$ that do not have greatest elements. We can
see that $\mathscr{D}$ is compact with respect to the induced order. By definition, $\mathscr{D}$
is also densely ordered. The fact that $G$ is densely ordered and not a singleton implies that
$\mathscr{D}$ is not the singleton $\{\varnothing\}$, and moreover that no non-empty
open subset of $\mathscr{D}$  can be a~singleton. Notice furthermore that $\mathscr{D}$ is first
countable because $G$ contains no copies of $\wone$ or $\wone^*$.
In particular, if $J \in \mathscr{D}$ is non-empty, then there is a strictly increasing sequence
$(J_n)_{n=1}^\infty$ in $\mathscr{D}$, having union $J$.\medskip

Mimicking a little the procedure in step one above, for every $I \in \mathscr{D}$,
define
\[
W_I \;=\; \bigcap_{f \in I} f^{-1}(0) \cap \bigcap_{g \in G\setminus I} g^{-1}(D),
\]
and take $\sigma_I \in \Sigma$ such that
\[
W_I \cap H_{\sigma_I} \;=\; W_I \cap \widehat{H}_{\sigma_I} \;\neq\; \varnothing.
\]
Let $\mathscr{T}_\sigma = \set{I \in \mathscr{D}}{\sigma_I=\sigma}$. As
$\mathscr{D} = \bigcup_{\sigma\in\Sigma} \mathscr{T}_\sigma$, the
Baire Category Theorem implies that, for some $\sigma$, the closure $\cl{\mathscr{T}_\sigma}$ contains
a non-empty open set $\mathscr{U}$. Since $\mathscr{U}$ cannot be a~singleton, it follows that $(P,Q)
\subseteq \cl{\mathscr{T}_\sigma}$ for some $P,Q \in \mathscr{U}$.\medskip

Fix $J \in \mathscr{T}_\sigma \cap (P,Q)$, take a strictly increasing sequence $(I_n)_{n=1}^\infty$ in $(P,J)$
having union $J$, and select $J_n \in \mathscr{T}_\sigma \cap (I_n,I_{n+1})$ for each $n$. As above, let
$\sigma = (n_1,n_2,\dots,n_k)$, $m=n_k$ and $A = \bigcap_{i<k} (\cl{H_{n_i}}\setminus H_{n_i})$.
Take $t_n \in W_{J_n} \cap A \cap \cl{H_m}$ for all $n$, $f_1 \in J_1$, and
$f_n \in J_n\setminus J_{n-1}$ for $n \geqslant 2$. Then $f_n(t_n)=0$ and $|f_{n+1}(t_n)|\geqslant 1$
for all $n$. Fix a limit $u \in A \cap \cl{H_m}$ of the $t_n$ and pick any $f \in G\setminus J$.
As above, according to Lemma \ref{1-step}, $u \in \cl{H_m}$ does not have a local base
of connected sets relative to $\cl{H_m}$, thus $u \notin H_m$.\medskip

However, we claim that $u \in W_J$, which is a contradiction because it implies that
\[
u \in W_J \cap A \cap \cl{H_m} = W_J \cap A \cap H_m.
\]
Indeed, given any $f \in G\setminus J$,
as $f_{n+1} \prec f$, we have $|f(t_n)| \geqslant 1$ for all $n$, whence $|f(u)|\geqslant 1$.
On the other hand, if $f \in J$ then there exists $N$ such that $f \in J_n$ for $n \geqslant N$.
Thus $f(t_n)=0$ for all such $n$ and so $f(u)=0$. Therefore $u \in W_J$ as claimed and
we have our desired contradiction. This completes the proof of assertion (1).\medskip

Assertions (2) and (3) follow easily. Suppose that $M$ has {\bish}. Let $F \subseteq C(K)$
be a~bounded $1$-$\prec$-chain. The fact that $M$ has {\bish} implies that there are
only countably many distinct $\sim_M$-equivalence classes. By assertion (1), every such equivalence
class of $F$ is countable, so it follows that $F$ itself must be countable. Finally, for assertion (3),
if $M$ is empty then $\sim_M$ has just one equivalence class, so $F$ is countable and scattered.
\end{proof}

We conclude this section by remarking that the converse of Theorem \ref{locconbishop} part (2) is false. The long
line $L$ is locally connected, so $L$ has {\bish}. Meanwhile, $[0,\wone] \subseteq L$,
$[0,\wone]$ does not have {\bish}, and every point of the dense remainder $L\setminus[0,\wone]$
has a local base of connected sets relative to $L$.

\section{Further observations}\label{further-observations}

The class of spaces having {\bish} lacks good permanence properties and, in particular, a~closed subset of space that has {\bish} need not have {\bish}. Indeed, the above example of the long line illustrates this. For another
example, take any compact, Hausdorff space $M$ not having {\bish}, apply Corollary~\ref{dualball}, and observe
the natural embedding of $M$ into $B_{C(M)^*}$ via the Dirac delta functionals.\medskip

Nonetheless, it is easy to see that if $K$ is a compact space that has {\bish}, $M$ is compact and $\map{\pi}{K}{M}$
is a continuous surjection, then $M$ has {\bish} too. This follows from the fact that $f \mapsto f \circ \pi$
is an isometry of $C(M)$ into $C(K)$ that respects the lattice structure.\medskip

Moreover, {\bish} is not preserved by Banach-space isomorphisms of $C(K)$-spaces.

\begin{example}\label{no-isomorphisms}
The property {\bish} is not preserved under Banach-space isomorphisms of $C(K)$-spaces.
\end{example}

\begin{proof}
Let $B=B_{C(\beta\N)^*}$ be the dual unit ball of $C(\beta\N)$, which is isometrically isomorphic to $\ell_\infty$.
It has {\bish} by Corollary~\ref{dualball}. By the Banach--Mazur Theorem, $C(\beta\N)$ embeds into $C(B)$ isometrically.
On the other hand, $C(\beta\N)$ is injective and isomorphic to its Cartesian square. We are now in a position to apply the Pe{\l}czy\'{n}ski decomposition method in order to conclude that there exists an~isomorphism $$C(B)\cong C(B)\oplus_\infty C(\beta\N).$$ On the other hand, the Banach spaces $C(B)\oplus_\infty C(\beta\N)$ and $C(B\sqcup
\beta\N)$ are isometrically isomorphic (here $\sqcup$ denotes disjoint union). Because $\beta\N$ fails {\bish},
$B\sqcup \beta\N$ fails it too.
\end{proof}

Finally, we return to the structure of the left ideal of operators on $C(K)$ having
{\bish}. In \cite[Question 4.3]{hkk:13}, the authors ask whether this ideal is
always two-sided, regardless of whether $K$ has {\bish} or not. We can use the
spaces of Example \ref{no-isomorphisms} to answer this question.

\begin{example}
The set of operators on $C(B\sqcup \beta\N)$ satisfying {\bish} is not a right
ideal.
\end{example}

\begin{proof}
Let $\map{S}{C(B)}{C(B\sqcup \beta\N)}$ be a Banach-space isomorphism. As the Banach spaces $C(B\sqcup \beta\N)$ and $
C(B) \oplus_\infty C(\beta\N)$ are isometrically isomorphic, we may extend $S$
to an operator $\map{T}{C(B\sqcup \beta\N)}{C(B\sqcup \beta\N)}$ by setting $T$ equal to $0$ on $C(\beta\N)$.
Note that $T$ has {\bish} because $S$, as an operator from $C(B)$, has {\bish}.
It remains to notice that $TS^{-1} = I_{C(B\sqcup \beta \mathbb{N})}$ fails {\bish}.
\end{proof}

\subsection*{Acknowledgements}

This paper grew out of discussions held at the Workshop on set theoretic methods
in compact spaces and Banach spaces, Warsaw, and at the inaugural meeting of
the QOP network, Lancaster University, April 2013.\medskip

In the Workshop on set theoretic methods in compact
spaces and Banach spaces, Warsaw, 2013, K.~P.\ Hart announced two results related to
this work:\ the \v Cech--Stone compactification of $[0,\infty)$ does not have {\bish}
(\emph{cf.} Example \ref{connectednobishop}), and a compact, locally connected space has {\bish}
(\emph{cf.} Theorem \ref{locconbishop}).\medskip

Finally, we thank the referee for his or her extremely thorough report, and for the many suggestions which helped to improve the presentation and clarity of the paper.

\end{document}